\documentclass{article}

\usepackage{amsmath}
\usepackage{amsfonts}
\usepackage{amssymb}
\usepackage{amsthm}
\usepackage{enumerate}
\usepackage{fullpage}

\newtheorem{theorem}{Theorem}[section]
\newtheorem{conjecture}{Conjecture}[section]
\newtheorem{lemma}[theorem]{Lemma}

\newcommand{\F}{\mathbb{F}}

\newcommand{\poly}{\mathsf{poly}}
\newcommand{\one}{{\mathbf{1}}}

\usepackage{parskip}

\begin{document}
\title{On the degree of polynomials computing square roots mod $p$}
\author{ Kiran Kedlaya\thanks{
Department of Mathematics, University of California San Diego.\\
Research supported by NSF grant DMS-2053473, the UC San Diego
Warschawski Professorship, the Simons Fellows in Mathematics program of
the Simons Foundation (2023--24 academic year), and the School of
Mathematics of the Institute for Advanced Study (2023--24 academic year).\\
Email: kedlaya@ucsd.edu
}
 \and Swastik Kopparty\thanks{
Department of Mathematics and Department of Computer Science, University of Toronto.\\
Research supported by an NSERC Discovery Grant.\\
Email: swastik.kopparty@utoronto.ca
}
}
\date{}
\maketitle

\begin{abstract}
For an odd prime $p$, we say $f(X) \in {\mathbb F}_p[X]$ {\em computes square roots} in $\mathbb F_p$ if, for all nonzero perfect squares $a \in \mathbb F_p$, we have $f(a)^2 = a$.
 
When $p \equiv 3 \mod 4$, it is well known that $f(X) = X^{(p+1)/4}$ computes square roots. This degree is surprisingly low (and in fact lowest possible), since we have specified $(p-1)/2$ evaluations (up to sign) of the polynomial $f(X)$.
On the other hand, for $p \equiv 1 \mod 4$ there was previously no nontrivial bound known on the lowest degree of a polynomial computing square roots in $\mathbb F_p$.

We show that for all $p \equiv 1 \mod 4$, the degree of a polynomial computing square roots has degree at least $p/3$. 
Our main new ingredient is a general lemma which may be of independent interest: powers of a low degree polynomial cannot have too many consecutive zero coefficients.
The proof method also yields a robust version: any polynomial that computes square roots for 99\% of the squares also has degree almost $p/3$. 

In the other direction, Agou, Deligl{\'e}se, and Nicolas~\cite{ADN-sqrt} showed that for infinitely many $p \equiv 1 \mod 4$, the degree of a polynomial computing square roots can be as small as $3p/8$. 
\end{abstract}

\section{Introduction}
Let $p$ be an odd prime, and let $\F_p$ be the finite field with $p$ elements.

We say $f(X) \in \F_p[X]$ {\em computes square roots} in $\F_p$, if
for all nonzero perfect squares $a \in \F_p$, we have:
$$f(a)^2 = a.$$ In other words, for each nonzero perfect square $a \in \F_p$, $f(a)$ is one of its two square roots.

When $p \equiv 3 \mod 4$, then it is well known that $f(X) = X^{(p+1)/4}$ computes square roots. This degree is surprisingly low, since we are essentially interpolating a polynomial from $(p-1)/2$ evaluations (where the evaluations are specified up to sign). We are interested in whether there is a similar phenomenon for $p \equiv 1 \mod 4$.

Concretely, we study the question: what is the smallest degree of a polynomial that computes square roots? Despite being a basic and natural question, there were no nontrivial bounds known for this question for the case of $p \equiv 1 \mod 4$. 

There is a very simple argument\footnote{This argument works for all $p$, and thus we get that $X^{(p+1)/4}$ is the lowest degree polynomial computing square roots for $p \equiv 3 \mod 4$.}  that shows that the degree of such a polynomial $f(X)$ must be at least $\frac{p-1}{4}$; indeed, the nonzero polynomial $f(X)^2 - X$ vanishes at all the $\frac{p-1}{2}$ nonzero perfect squares in $\F_p$. 

Our main result is that, unlike the case of $p \equiv 3 \mod 4$, the degree of any polynomial computing square roots in the case of $p \equiv 1 \mod 4$ must be significantly higher, about $\frac{1}{3}\cdot p$.
\begin{theorem}
\label{thm:1mod4}
 Let $p \equiv 1 \mod 4$. Then any polynomial that computes square roots in $\F_p$ has degree at least $\frac{p-1}{3}$.
\end{theorem}
Our proof is based on expressing the property of computing square roots as a polynomial relation (involving some unknown polynomials), and then eliminating the unknown polynomials through a combination of taking derivatives and truncations. Abstracting out the main steps, we get a general lemma (a special case of the Mason-Stothers {\em abc}-theorem) which may be of independent interest: the powers of a low-degree polynomial cannot have too many consecutive zero coefficients.

How does  $p \mod 4$ play a role in the proof? Our proof ends up showing that for all $p$, any polynomial $f(X)$ of degree less than $\frac{p}{3}$ that computes square roots must have $f(X)^2 = X^{(p+1)/2}$ (as a polynomial identity), and this is not possible if $p \equiv 1 \mod 4$.

\subsubsection*{A robust version}

The degree of a polynomial computing a certain function is quite a brittle notion. Changing just a single value of the function can change the degree drastically. By using the key idea of the Berlekamp-Welch algorithm for decoding Reed-Solomon codes, we can strengthen the above result to get a robust version, given below.

\begin{theorem}
\label{thm:1mod4robust}
 Let $p \equiv 1 \mod 4$. Then any polynomial that computes square roots in $\F_p$ on all but $e \leq \frac{p-7}{12}$ nonzero perfect squares in $\F_p$ must have degree at least $\frac{p-1}{3} - e$.
\end{theorem}

The connection to decoding algorithms for Reed-Solomon codes is not such a surprise. The problem of whether a low-degree polynomial can compute square roots is in fact a list-recovery problem for Reed-Solomon codes~\cite{RSlist}; our result effectively shows that a certain algebraic list recovery instance where each input list has size $2$ has no solutions. The difficulty is that this lies beyond the regime where we have a good understanding of list-recoverability and list-decodability of Reed-Solomon codes.

More concretely, let $C$ be the the Reed-Solomon code of degree $d$ polynomials over $\F_p$ with evaluation set $D$.
Suppose for each $x \in D$ we are given a set $S_x \subseteq \F_p$ with $|S_x| \leq 2$. How can we certify that there are no codewords $c$ of $C$ such that for each coordinate $x \in D$, we have $c_x \in S_x$? 
It is not known how to give an efficiently verifiable certificate of this {\em in general} when $d = \left(\frac{1}{2} + \Omega(1) \right) |D|$. In our setting $D$ is the set of perfect squares (so $|D| = (p-1)/2$), and $d$ is $p/3$, which is outside the range of known certification methods~\cite{RSlist}.

\subsubsection*{Better upper bounds for special $p$}

It turns out that for some $p$ which are $1 \mod 4$, there are polynomials computing square roots with degree about $\frac{3}{8} \cdot p$.

\begin{theorem}
\label{thm:5mod8}
 Let $p \equiv 5 \mod 8$. Then there is a polynomial that computes square roots in $\F_p$ with degree at most $\frac{3p+1}{8}$.
\end{theorem}

This was first shown by Agou, Deligl{\'e}se and Nicolas~\cite{ADN-sqrt} (see also~\cite{Carella-sqrt}). In the first version of this paper posted online, we gave a proof of this, unaware of their result. The proof that we had found of this theorem is the same as that of~\cite{ADN-sqrt} -- based on the Tonelli-Shanks algorithm -- but since it is quite simple and short, we keep it in this updated paper for completeness.


The method of Tonelli-Shanks also yields similar phenomena with degree $(\frac{1}{2} - \Omega(1))p$ for $p$ in special residue classes mod $2^j$ with $j$ constant. Theorem 5 of \cite{ADN-sqrt} gives another example of such a phenomenon for $p \equiv 7 \mod 12$, giving polynomials (different from $X^{(p+1)/4}$) computing square roots of degree about $\frac{5}{12} \cdot p$.

\subsubsection*{Upper bounds for general $p$}

Finally, we discuss upper bounds for the case of general $p$.
First, a heuristic. There are $2^{(p-1)/2}$ different square root functions (the choice of sign for each perfect square). If the unique interpolating polynomials of degree $< (p-1)/2$ for these functions had their coefficients behaving randomly, then we would expect a polynomial of degree at most $\frac{1}{2}p - \Omega\left( \frac{p}{\log p} \right)$ that computes square roots.

Formalizing this intuition, we show that there is a polynomial computing square roots with degree $\frac{1}{2} p - \widetilde{\Omega}(\sqrt p)$. This is done by looking at the explicit formulas for the coefficients of the interpolants and arguing their pseudorandomness via the Weil bounds and some elementary Fourier analysis.

We also note that all our results have analogues for computing $t$th roots.

\subsubsection*{Related Work}

The work of Agou, Deligl{\'e}se and Nicolas~\cite{ADN-sqrt} gave examples, for infinitely many $p$, of polynomials in $\F_p$ of abnormally low degree that compute square roots. The focus there was on finding polynomials with few monomials, and they gave interesting upper and lower bounds for this. Chang, Kim and Lee~\cite{CKL-nthroot} gave analogues of these results for computing $t$th roots.

Another related line of research has studied lower bounds on the degree of polynomials computing interesting arithmetic functions.
Coppersmith and Shparlinski~\cite{CS-dlog} (following an error-free computation result of Mullen and White~\cite{MW-dlog}), gave strong lower bounds on the degree of polynomials computing the discrete logarithm in prime fields $\F_p$ with as many as $(1-o(1))p$ errors. Winterhof~\cite{W-dlog} later gave a generalization of this to all finite fields. These results are related to {\em list-decoding} of Reed-Solomon codes, since for each input there is only one ``correct'' value which we are hoping the polynomial will compute. As mentioned earlier, the problems we consider are related to {\em list-recovery} of Reed-Solomon codes, where there are multiple ``correct'' values for any given input, and we hope the polynomial computes one of them.

\subsubsection*{Conclusions and Questions}

Computing square roots and understanding quadratic residuosity are central topics in algebraic computation and pseudorandomness.

Perhaps the most interesting and fundamental open question in this area is that of deterministically computing square roots mod $p$ in $\poly(\log p)$ time.
As we already saw, when $p \equiv 3 \mod 4$, the simple deterministic $\poly(\log p)$ algorithm of raising $x \in \F_p$ to the power $\frac{p+1}{4}$ computes the square root of $x$. Our results show that the $p \equiv 1 \mod 4$ case is qualitatively different in some respects. 
See~\cite{BachShallit,GerhardGathen,Shoup} for what is known about this computational problem and related number theoretic issues.

Other important questions include the problem of determining the size of the least quadratic residue mod $p$ (this is also connected to deterministic computation of square roots), and understanding the pseudorandomness of the Paley graph (for example, are Paley graphs Ramsey graphs?).

Finally, as mentioned above, our results can be viewed as showing that a certain list-recovery instance for Reed-Solomon codes has no solutions. 
We close with a conjecture about the list-recoverability of Reed-Solomon codes.
The conjecture talks about prime fields; the results of Guruswami and Rudra~\cite{GR-listrecovery} imply that this assumption cannot be dropped.

\begin{conjecture}
 Let $\F_p$ be a prime field. Let $\ell \in \mathbb N, \epsilon >0$ be constants. Suppose we are given, 
 for each $x \in \F_p$, a set $S_x$ with $|S_x| \leq \ell$. 
 Then:
 $$ \left|\{ P(X) \in \F_p[X] \mid \deg(P) \leq (1-\epsilon) p, \mbox{ and for all $x \in \F_p$, } P(x) \in S_x \}\right| \leq p^{O_{\epsilon, \ell}(1)}.$$
\end{conjecture}

We hope that our methods can give some insight into understanding the list-recovery capacity of Reed-Solomon codes, and in particular the above conjecture.

\section{Lower bound for polynomials computing square roots}

We now prove our first theorem about polynomials computing square roots mod $p$.

\noindent{\bf Theorem \ref{thm:1mod4}.\ }
{\em Let $p \equiv 1 \mod 4$. Then the degree of any polynomial that computes square roots in $\F_p$ is at least $\frac{p-1}{3}$.}
\begin{proof}
Suppose $f(X)$ is of degree $d < \frac{p-1}{3}$ and computes square roots in $\F_p$.
Then, since $X^{(p-1)/2} - 1$ is the vanishing polynomial of the set of nonzero perfect squares in $\F_p$, we have:
$$ f(X)^2 - X   \equiv 0  \mod (X^{(p-1)/2} - 1).$$

Let $A(X)$ be the polynomial of degree $2d-(p-1)/2$ such that
$$f(X)^2 - X =  A(X) \cdot (X^{(p-1)/2} - 1).$$
Let $B(X) = X - A(X)$.
Then we get:
\begin{equation}
\label{eq1}
f(X)^2 = A(X) \cdot X^{(p-1)/2}  + B(X),
\end{equation}
where:
\begin{itemize}
\item $\deg(f(X)) = d$.
 \item $\deg(A(X)), \deg(B(X)) \leq 2d - (p-1)/2$.
 \item $A(X) \neq 0$. If $A(X) = 0$ then $f(X)^2 = X$, which is impossible for a polynomial $f(X)$. 
 \item $B(X) \neq 0$. Otherwise $A(X) = X$, and $f(X)^2 = X^{(p+1)/2}$, which is possible only if $p \equiv 3 \mod 4$.
\end{itemize}

These conditions together will give us our lower bound on $d$.

Taking derivatives\footnote{Throughout this paper, we work with formal derivatives of polynomials.} of both sides of~\eqref{eq1}, we get:
\begin{align}
\label{eq2}
2f(X)f'(X) = A'(X) \cdot X^{(p-1)/2} - \frac{1}{2}A(X) X^{(p-3)/2} + B'(X).
\end{align}

Computing $2 f(X)^2 f'(X)$ in two ways using~\eqref{eq1} and~\eqref{eq2}, we get:
$$ 2 f'(X) A(X) X^{(p-1)/2} + 2 f'(X) B(X) = f(X) \left(X \cdot A'(X) - \frac{1}{2}A(X)\right) X^{(p-3)/2} + f(X) B'(X).$$
Now, using our assumption on $d$, the degrees of $2 f'(X)B(X)$ and $f(X) B'(X)$ are both at most
$3d - (p-1)/2 - 1 < (p-3)/2$, and thus taking the above equation mod $X^{(p-3)/2}$,
$$ 2f'(X) B(X) = f(X) B'(X).$$

Since $B(X) \neq 0$, we get $\frac{2 f'(X)}{f(X)} = \frac{B'(X)}{B(X)}$. Since $2\deg(f), \deg(B) < p$, by a basic property of logarithmic derivatives, this implies $f(X)^2 = \lambda B(X)$ for some nonzero $\lambda$, contradicting the fact that $A(X) \neq 0$. (See Remark 1 in Section~\ref{sec:rem} for a precise statement and a proof.)

Thus our assumption that $d < \frac{p-1}{3}$ is wrong, and the theorem follows.
\end{proof}

\subsection{Consecutive zero coefficients in powers of polynomials}

We isolate the key step above as the following lemma:

\begin{lemma}
\label{lemgen}
Let $\F$ be a field of characteristic $p$.
Let $f(X), A(X), B(X)$ be in $\F[X]$. 
 Suppose
 \begin{equation}
 \label{geneq1}
    f(X)^t = A(X) \cdot X^\ell + B(X)
 \end{equation}
 where:
 \begin{itemize}
  \item $\deg(f(X)) \leq d < \frac{p}{t}$,
  \item $\deg(B(X)) \leq b < p$,
  \item $A(X) \neq 0$,
  \item $B(X) \neq 0$,
\end{itemize}
Then $d + b \geq \ell$.
\end{lemma}
In words, this says that if $dt < p$, the $t$th power of a polynomial of degree $d$ cannot have $d$ consecutive $0$ coefficients.
\begin{proof}
Supose $d + b < \ell$. Observe that this implies that $f(X) \neq 0$.

 Taking derivatives of both sides of \eqref{geneq1}, we get:
 \begin{equation}
  \label{geneq2}
  t f(X)^{t-1} f'(X) = C(X) X^{\ell-1} + B'(X),
 \end{equation}
 for some $C(X) \in \F[X]$. 
 
 Computing $t f(X)^t f'(X)$ in two different ways using~\eqref{geneq1},~\eqref{geneq2}, we get:
 $$ t A(X) f'(X) X^{\ell} + t f'(X) B(X) = f(X) C(X) X^{\ell-1} + f(X) B'(X).$$
 Since $\deg(t f'(X) B(X)), \deg(f(X) B'(X)) < d+b \leq \ell - 1$,
 by taking this equation mod $X^{\ell-1}$ we get:
 $$ t f'(X) B(X) = f(X) B'(X),$$
 and since $f(X), B(X)$ are nonzero, we get that:
 $$ t \frac{f'(X)}{f(X)} = \frac{B'(X)}{B(X)}.$$
 
 By the logarithmic derivative, we get $f(X)^t = \lambda B(X)$ for some nonzero $\lambda$, contradicting our assumption that $A(X) \neq 0$.

 Thus $d+ b \geq \ell$ as claimed.
\end{proof}

\subsection{Remarks}
\label{sec:rem}

\begin{enumerate}
 \item The key fact about logarithmic derivatives that we are using is that if 
 $f(X), B(X) \in \F_p[X]$, $t \cdot \deg(f), \deg(B) < p$, and:
 $$ t \frac{f'(X)}{f(X)} = \frac{B'(X)}{B(X)},$$
 then $f(X)^t = \lambda \cdot B(X)$ for some constant $\lambda \in \F_p$. 
 
 We recap a quick proof. The hypothesis implies that $\left( \frac{f(X)^t}{B(X)} \right)' = 0$,
  and thus $\frac{f(X)^t}{B(X)}$ must be a rational function in $X^p$. To see the last deduction, note that $\left(f(X)^t \cdot B(X)^{p-1} \right)' = \left( \frac{f(X)^t}{B(X)} \cdot B(X)^p \right)' = \left( \frac{f(X)^t}{B(X)} \right)' \cdot B(X)^p + \frac{f(X)^t}{B(X)} \cdot p \cdot B(X)^{p-1} \cdot B'(X) = 0 + 0 = 0$. Thus the polynomial $f(X)^t \cdot B(X)^{p-1}$ is a polynomial in $X^p$, which implies that $\frac{f(X)^t}{B(X)} = 
 \frac{f(X)^t \cdot B(X)^{p-1}}{B(X)^p} =
 \frac{f(X)^t \cdot B(X)^{p-1}}{B(X^p)}$ is a rational function in $X^p$. 
 Once we know that $\frac{f(X)^t}{B(X)}$ is a rational function in $X^p$, our assumption about the degrees implies the result.

 \item The exact same proof also classifies when low-degree polynomials can compute square roots of very-low-degree polynomials on a multiplicative group.
 \begin{theorem}
 Let $G \subseteq \F_{p}^*$ be a multiplicative subgroup of size $m$, with $m \leq \frac{p-1}{2}$.
  Let $C(X) \in \F_p[X]$ have degree at most $\frac{m}{3}$.
Suppose $f(X) \in \F_p[X]$ is such that $f(a)^2 = C(a)$ for all $a \in G$.
Then one of the following alternatives must hold:
 \begin{itemize}
  \item $f(X)^2 = C(X)$,
  \item $f(X)^2 = C(X) \cdot X^m$,
  \item $\deg(f) \geq \frac{2m}{3}$. 
 \end{itemize}
 \end{theorem}

 The above statement for $m = p-1$ and $C(X)$ being a constant follows from a result of Biro~\cite{Biro}, who classified low-degree polynomials that take two values on $\F_p^*$. The proof from~\cite{Biro} is a delicate investigation of certain power sums. Our proof for $m < p-1$ is quite different, and has the flexibility of allowing for the robust version proved in the next section (which gives, for example, a classification of low-degree polynomials that take only 2 values on $99\%$ of $G$). 
 
 In this generality, the bound of $\frac{2m}{3}$ is tight.
 If $m$ is divisible by $3$, then the polynomial $f(X) = \left(X^{2m/3} + X^{m/3} - \frac{1}{2}\right)$ is such that $f(x)^2 = \frac{9}{4}$ for all $x \in G$ (since $x^{m/3}$ is a cube root of $1$).

 \item The proof of Lemma~\ref{lemgen} also applies as is to {\em rational powers} of $f(X)$, where we now talk about consecutive 0 coefficients in the power series. We only state it for characteristic $0$; it says that the power series expansion of $f(X)^{r/s}$, for $f(X)$ of degree $d$, does not have $d$ consecutive $0$ coefficients. Precisely, we have:
  
\begin{lemma}
\label{lemgen2}
Let $\F$ be a field of characteristic $0$. 
Let $t$ be a rational number.
Let $f(X) \in \F[X]$ be a polynomial of degree at most $d$.

Then any (formal) power series expansion of $f(X)^t$ in $\F[[X]]$ does not have
$d$ consecutive zero coefficients.
\end{lemma}

This is stronger than the usual bound for this situation
(which shows up in polynomial factoring algorithms via the Hilbert irreducibility theorem~\cite{KaltofenHIT} and the Arora-Sudan low degree test~\cite{ASlowdegree}), 
which goes as follows:
Suppose $f(X)^{1/s} =  A(X) X^\ell + B(X)$, where $\deg(f) = d, \deg(B) = b$ and $A(X) \in \F[[X]]$ is nonzero, then $f(X) - B(X)^s$ is a nonzero polynomial of degree at least $\ell$, and so $\ell \leq \max(sb, d)$. Thus
if $b$ is large, this bound only guarantees that there is a nonzero coefficient 
$X^i$ for $i \in [b+1, sb]$  (instead of $[b+1, b+d]$ as guaranteed by Lemma \ref{lemgen2}).

\item Applying the same method, we can apply this method to the power series expansion of $e^{f(X)}$ too.

\begin{lemma}
\label{lemgen3}
Let $\F$ be a field of characteristic $0$. 
Let $f(X) \in \F[X]$ be a polynomial of degree at most $d$ with constant term $0$.

Then the (formal) power series expansion of $e^{f(X)}$ in $\F[[X]]$ does not have $d$ consecutive zero coefficients.
\end{lemma}
 
 \item The bounds of Lemma~\ref{lemgen}, Lemma~\ref{lemgen2} and Lemma~\ref{lemgen3} on the number of consecutive $0$ coefficients are tight, for example when $f(X)$ is of the form $\alpha X^d + \beta$.
 
 \item We can give another proof of Lemma~\ref{lemgen} (but not Lemma~\ref{lemgen2} or Lemma~\ref{lemgen3} as far as we know) using the Mason-Stothers {\em abc}-theorem for polynomials~\cite{abcthm1,abcthm2}. 
 
 Indeed, note that $f(X)^t$, $A(X) \cdot X^\ell$ and $B(X)$ all have degree at most $dt$. Furthermore, the radical of their product divides $f(X) \cdot A(X) \cdot X \cdot B(X)$, and thus has degree at most 
 $d + (dt-\ell) + 1 + b$. By the {\em abc}-theorem, we get that:
 $$ dt \leq (d + dt - \ell + 1 + b) - 1 = dt + d - \ell + b,$$
 and thus $d + b \geq \ell$. 
 
\end{enumerate}

\section{A robust version}

Let $p$ be a prime that is $1$ mod $4$. 
Let $S$ be the set of nonzero perfect squares in $\F_p$.

We say a polynomial $f(X) \in \F_p[X]$ {\em computes square roots with error $e$} if:
$$ \left|\{a \in S \mid f(a)^2 \neq a  \}\right| \leq e.$$

We show that any polynomial computing square roots even allowing $\Omega(p)$ error cannot have degree much smaller than $p/3$.

\noindent{\bf Theorem \ref{thm:1mod4robust}.\ }
{\em Let $p \equiv 1 \mod 4$. Suppose $f(X) \in \F_p[X]$ is a polynomial of degree $d$ that computes square roots with error $e$.

Then
$$ d \geq \begin{cases}
    \frac{p-1}{3}  - e &    e \leq \frac{p-7}{12}\\
    \frac{p-1}{2} - 3e - 1 &  e > \frac{p-7}{12}.
          \end{cases}
$$}
\begin{proof}
 We use the idea of the Berlekamp-Welch Reed-Solomon decoding algorithm~\cite{BerlekampWelch}.
 
 Let $U \subseteq S$ be the set of $a \in S$ where $f(a)^2 \neq a$.
 
 Let $E(X) \in \F_p[X]$ be the vanishing polynomial of $U$, given by:
 $$E(X) = \prod_{u \in U} (X - u).$$
 Note that $E$ is a nonzero polynomial of degree at most $e$.
 
 Then we have:
 $$ E(X)^2 \cdot f(X)^2  \equiv  E(X)^2 \cdot X \mod (X^{(p-1)/2} - 1).$$
 
 Let $A(X)$ be the polynomial of degree at most $2(e+d)-(p-1)/2$ such that:
 $$ E(X)^2 \cdot f(X)^2 - E(X)^2 \cdot X = A(X) \cdot (X^{(p-1)/2} - 1).$$ 
 
 Let $g(X) = E(X) \cdot f(X)$, and $B(X) = E(X)^2 \cdot X - A(X)$.
 
 Then
 $$g(X)^2 = A(X) \cdot X^{(p-1)/2}  +  B(X).$$
 
 We have:
 \begin{itemize}
  \item $\deg(g) \leq d+e$,
  \item $\deg(B) \leq \max(2e+1, 2(e+d) - (p-1)/2)$.
  \item $A(X) \neq 0$. Otherwise $E(X)^2 \cdot f(X)^2 = E(X)^2 \cdot X \Longrightarrow f(X)^2 = X$, which is impossible.
  \item $B(X) \neq 0$. Otherwise $E(X)^2 \cdot X = A(X)$, and so 
  $E(X)^2 f(X)^2 =  E(X)^2 X^{(p+1)/2}$, which implies that $f(X)^2 = X^{(p+1)/2}$. This is only possible if $p \equiv 3 \mod 4$.
 \end{itemize}
 Plugging this into Lemma~\ref{lemgen}, we get:
 \begin{itemize}
  \item  $$(d+e) + \left(2(d+e) - \frac{p-1}{2}\right) \geq \frac{p-1}{2},$$
  if $2d - (p-1)/2 \geq 1$,
  \item $$ (d+e) + (2e +1) \geq \frac{p-1}{2},$$
  if $2d - (p-1)/2 \leq 0$.
 \end{itemize}

 This tells us that either:
 $$ d \geq \frac{p-1}{3} - e,$$
 or:
 $$ d \geq \frac{p-1}{2} - 3e - 1,$$
 and thus:
 $$ d \geq \min\left( \frac{p-1}{3} - e, \frac{p-1}{2} - 3e-1 \right),$$
 which gives us the desired claim.
\end{proof}

Note that there is another simple lower bound (which applies for all $p$) of $d + \frac{e}{2} \geq \frac{p-1}{4}$ (the simple lower bound is better for $e > \frac{p-1}{10}$).
This is proved by considering the number of roots of the degree $2d$ polynomial $f(X)^2 - X$.

\section{Upper bound for special $p$}

In this section, we give an upper bound on the degree of polynomials computing square roots mod $p$, for infinitely many $p \equiv 1 \mod 4$. The upper bound is best when $p \equiv 5 \mod 8$, and we only present this case. The result and proof of this section is due to Agou, Deligl{\'e}se and Nicolas~\cite{ADN-sqrt}. It remains in this paper only for completeness.

\noindent{\bf Theorem \ref{thm:5mod8}.\ }
{\em  Let $p \equiv 5 \mod 8$. Then there is a polynomial that computes square roots in $\F_p$ with degree at most $\frac{3p+1}{8}$.}
\begin{proof}
Since $p \equiv 1 \mod 4$, we get that $-1$ is a perfect square mod $p$. Let $i \in \F_p$ be one of the square roots of $-1$.
Our main ingredient is the Tonelli-Shanks algorithm~\cite{ShanksTonelli1,ShanksTonelli2} computing square roots mod $p$. For $p = 4 \ell + 1$, the algorithm essentially gives a formula for the square root depending on two cases.
Specifically, let $u: S \to \F_p$ given by:
$$ u(a) = \begin{cases} 
            a^{(p + 3)/8} & a^{(p-1)/4} = 1, \\
            i \cdot a^{(p + 3)/8} & a^{(p-1)/4} = -1.
          \end{cases}$$
Then for all $a \in \F_p$, $u(a)$ is a square root of $a$.

This is already quite special; the set $S$ is partitioned into two equal sized parts $S_0$ and $S_1$, and on each $S_i$ we have a polynomial $f_i(X)$ computing the square root of degree about $\frac{1}{2} |S_i|$. (This is the lowest possible degree, since $f_i(X)^2 - X$ is a nonzero polynomial that vanishes on all of $S_i$.)

Usually if we have this kind of setup, even though the $f_i$ have unusually low degree, the unique polynomial $f$ (obtained from the Chinese remainder theorem) which restricts to $f_i$ on $S_i$ has no reason to have unusually low degree. 
But in this case it does!

Using the usual Chinese remainder formula, we consider the polynomial $f(X) \in \F_p[X]$ given by:
\begin{align*}
 f(X) &= \frac{1}{2} \left( X^{(p+3)/8} (X^{(p-1)/4} + 1) - i \cdot X^{(p+3)/8} (X^{(p-1)/4} - 1)  \right) \\
 &= \frac{1-i}{2}  X^{(3p+1)/8} + \frac{1+i}{2} X^{(p+3)/8}.
\end{align*}
By design, we have $f(a) = u(a)$ for all $a \in S$. 
Finally, notice that $\deg(f) \leq \frac{3p+1}{8}$. 

As a sanity check, we directly verify that $f(X)^2 \equiv X \mod (X^{(p-1)/2} -1)$.
Indeed,
\begin{align*}
f(X)^2 &=  \left(\frac{1-i}{2} \right)^2 X^{(3p+1)/4} + 2 \cdot \frac{(1-i)(1+i)}{4} X^{(4p+4)/8} + \left(\frac{1+i}{2} \right)^2 X^{(p+3)/4}\\
&= -\frac{i}{2} X^{(3p+1)/4} + X^{(p+1)/2} + \frac{i}{2} X^{(p+3)/4}\\
&=   \left(-\frac{i}{2} X^{(p+3)/4} +  X\right) \cdot (X^{(p-1)/2} - 1) + X,
\end{align*}
as desired.
\end{proof}

\section{Upper bounds for general $p$}

In this section, we give a slightly nontrivial upper bound on the degree of polynomials computing square roots for all $p$.
We will show that there is a polynomial with degree somewhat less than $\frac{p}{2}$ which computes square roots.

\begin{theorem}
 For all odd primes $p$, for $t = o\left(\frac{\sqrt{p}}{\log^2 p} \right)$, there is a polynomial of degree $\frac{p}{2} - t$ which computes square roots.
\end{theorem}
\begin{proof}
Let $m = (p-1)/2$. Let $S \subseteq \F_p$ be the set of nonzero perfect squares,
and note that $|S| = m$.
For each $\alpha \in S$, let $\delta_\alpha(X) \in \F_p[X]$ be
the unique polynomial of degree $\leq (m-1)$ such that for all $\beta\in S$:
$$ \delta_\alpha(\beta) =  \begin{cases} 1  & \beta = \alpha, \\ 0 & \beta \neq \alpha. \end{cases}$$
Explicitly, we have:
$$\delta_\alpha(X) = \frac{1}{m} \left( \left(\frac{X}{\alpha}\right)^{m-1} + \left(\frac{X}{\alpha}\right)^{m-2} + \ldots + \frac{X}{\alpha} + 1 \right).$$

Then given a function $u: S \to \F_p$, 
the unique polynomial $f(X) \in \F_p[X]$ of degree at most $m-1$
such that $f(\alpha) = u(\alpha)$ for all $\alpha \in S$ is given by:
$$f(X) =  \sum_{\alpha \in S} u(\alpha) \delta_\alpha(X).$$

Our goal is to pick $u$ where each $u(\alpha)$ is one of the two square roots of $\alpha$ so that many of the leading coefficients of $f(X)$ equal $0$.

We now use the structure of $S$. Let $g$ be a generator of $\F_p^*$. 
Then $S = \{g^{2j} \mid  0 \leq j < (p-1)/2 \}$. Furthermore, 
for $\alpha = g^{2j} \in S$, one of the two square roots of $\alpha$
is $g^{j}$.

Thus, our problem can be reformulated as choosing a function
$v: S \to \{\pm 1 \}$ such that:
$$ f(X) = \sum_{j=0}^{(p-1)/2} v\left(g^{2j}\right)\cdot g^j \cdot \delta_{g^{2j}}(X)$$
has many leading coefficients equal to $0$.

Observe that the coefficient of $X^{m-i}$ in $f(X)$ equals:
$$ \frac{1}{m} \sum_{j=0}^{(p-1)/2} v(g^{2j}) g^{j} \left( \frac{1}{g^{2j}} \right)^{m-i} = 
 \frac{1}{m} \sum_{i=0}^{(p-1)/2} v(g^{2j}) g^{(2i+1)j}.$$

 Thus, to get a polynomial $f(X)$ of degree $< m-t$, we want to find a vector $v \in \{\pm 1\}^m$ that lies in the kernel of the Vandermonde-type matrix $M \in \F_p^{t \times m}$, where:
 $$ M_{i,j} = g^{(2i+1)j}.$$
 (The row index $i$ runs from $1$ to $t$, the column index $j$ runs from $0$ to $m-1$.)
 
 For later use, for a vector $y \in \F_p^t$, we define $P_y(Z) \in \F_p[Z]$ to be the polynomial:
 $$ P_y(Z) = \sum_{i = 1}^{t} y_i Z^{2i+1}.$$
 Thus for $j \in \{0, 1, \ldots, m-1 \}$, the $j$th entry of $M^Ty$ equals $P_y(g^{j})$.

 To show that there exists the desired $\pm 1$ vector, we count the number of such vectors using Fourier analysis. Let $\omega$ be a $p$th root of unity in $\mathbb C$. The number of such $\pm 1$ vectors equals:
 \begin{align*}
 N &= \sum_{v \in \{\pm 1\}^m } \one_{M v = 0}\\
 &= \sum_{v \in \{\pm 1\}^m }  \mathbb E_{y \in \F_p^t}  \left[ \omega^{\langle y, Mv \rangle} \right]  \\
 &= \mathbb E_y \left[ \sum_v  \omega^{\langle M^T y, v \rangle}   \right]\\
 &= \mathbb E_y \left[ \sum_v  \omega^{\sum_{j=0}^{m-1} (M^T y)_j \cdot  v_j}   \right]\\
 &= \mathbb E_y \left[ \sum_v   \prod_{j=0}^{m-1} \omega^{(M^T y)_j \cdot  v_j}   \right]\\
 &= \mathbb E_y \left[ \sum_v \prod_{j=0}^{m-1} \omega^{P_y(g^{j}) \cdot  v_j} \right].
 \end{align*}
 

For $y = 0$, the expression inside the expectation equals $2^m$.
We will show that for the remaining $p^t -1$ values of $y$,
the expression inside the expectation is very small. 

Fix any $y \neq 0$. 
The expression inside the expectation equals:
\begin{align}
 \label{weilval}
\sum_{v \in \{\pm 1\}^m} \prod_{j=0}^{m-1} \omega^{P_y(g^{j}) \cdot  v_j} = \prod_{j=1}^m \left(\omega^{P_y(g^{j})} + \omega^{-P_y(g^{j})} \right).
\end{align}

The next lemma (which uses the Weil bounds on mixed character sums) shows that for any nonzero $y$, the evaluations of the polynomial $P_y$ at $\{1, g, g^2, \ldots, g^{m-1}\}$ are well distributed in $\F_p$.
\begin{lemma}
\label{lem:equidist}
Let $0 \leq \alpha <  \beta \leq 1$.
 Let $y \in \F_p^t \setminus \{0\}$. Then:
 $$ \Pr_{j \in \{0,1, \ldots, m-1\}}\left[ P_y(g^j) \in [\alpha p, \beta p] \right] = (\beta - \alpha )  + O\left(\frac{t\log^2 p}{\sqrt{p}} \right).$$
\end{lemma}

Assuming the lemma, we get that for $t = o\left(\frac{\sqrt{p}}{\log^2 p}\right)$, the product in Equation \eqref{weilval} is at most $2^m \cdot \exp(-m)$.
Thus:
\begin{align*} N &\geq \frac{2^m}{p^t}  - \max_{y \neq 0} \left| \prod_{j=1}^m \left(\omega^{P_y(g^{j})} + \omega^{-P_y(g^{j})} \right) \right|\\
 &\geq \frac{2^m}{p^t}  - O(2^{m} \exp(-m))\\
 &\geq 2^m \left(\frac{1}{p^t} - \exp(-m)  \right)\\
 & > 0,
\end{align*}
where the penultimate inequality holds since
$$t = o\left( \frac{\sqrt{p}}{\log^2 p} \right) \ll \frac{p}{\log p} = \Theta\left( \frac{m}{\log p} \right).$$

This completes the proof.
\end{proof}

\subsection{Distribution of values of $P_y(g^j)$}

We now prove Lemma~\ref{lem:equidist}.

\begin{proof}
 Let $I$ be the interval $[\alpha p, \beta p]$. Let $J$ be the set
 $\{1, g, g^2, \ldots, g^{m-1}\}$. Then the probability in the statement of the lemma equals:
 \begin{align}
 \label{pval}
 \frac{1}{m} \sum_{ z \in \F_p}  {\one}_{I}(P_y(z)){\one}_J(z).
 \end{align}
 
 $I$ is an interval in the additive group of $\F_p$
 By standard results, if we expand $\one_I$ in its additive Fourier expansion:
 $$ \one_I = \sum_\psi \widehat{\one}_I(\psi) \psi$$
 (where the $\psi$ are the additive characters), then:
 \begin{align}
 \label{l1bound1}
 \sum_{\psi } |\widehat{\one}_I(\psi)| \leq O( \log p).
 \end{align}
 
 Similarly, $J$ is an interval in the multiplicative group of $\F_p$.
 If we expand $\one_J$ in its additive Fourier expansion:
 $$ \one_J = \sum_\psi \widetilde{\one}_J(\chi) \chi$$
 (where the $\chi$ are the multiplicative characters), then:
 \begin{align}
 \label{l1bound2}
 \sum_{\chi} |\widetilde{\one}_I(\chi)| \leq O( \log p).
 \end{align}
 
 Then the probability in Equation~\eqref{pval} equals:
 \begin{align*} &\frac{1}{m}\sum_{z} \left( \sum_{\psi} \widehat{\one}_I(\psi) \psi(P_y(z)) \right) \left( \sum_{\chi} \widetilde{\one}_J(\chi) \chi(z) \right)\\
  &= \frac{1}{m}\sum_{\psi, \chi} \widehat{\one}_I(\psi) \widetilde{\one}_J(\chi) \left(\sum_{z} \psi(P_y(z)) \chi(z) \right)\\
  &= \frac{1}{m} \left( \frac{|I|}{p}\cdot \frac{|J|}{p} \cdot p + O\left(\sum_{(\psi, \chi) \neq (1,1)} |\widehat{\one}_I(\psi)|\cdot |\widetilde{\one}_J(\chi)| \cdot \left|\sum_{z} \psi(P_y(z)) \chi(z) \right|\right) \right)\\
  &= (\beta-\alpha) + \frac{1}{m} \cdot  O\left(\sum_{(\psi, \chi) \neq (1,1)} |\widehat{\one}_I(\psi)| \cdot | \widetilde{\one}_J(\chi) | \cdot \left|\sum_{z} \psi(P_y(z)) \chi(z) \right|\right) \\
 \end{align*}
The Weil bound for mixed character sums (see~\cite{Schmidtbook}) shows that whenever
$(\psi, \chi)$ are not both trivial characters, the inner expression is bounded:
$$\left|\sum_{z} \psi(P_y(z)) \chi(z) \right| \leq O( t \sqrt{p}).$$
Combined with the bounds in Equations~\eqref{l1bound1}, \eqref{l1bound2}, we get the desired bound on the probability.
\end{proof}

\section{$t$th roots}

Now we discuss $t$th roots in place of square roots. We think of $t$ as a constant, and the prime $p$ growing. Let $p \equiv 1 \mod t$, so that the set of nonzero $t$th powers in $\F_p$ has size $\frac{p-1}{t}$.

Just like in the case $t = 2$, for special $p$ there is a simple formula for computing the $t$th root; when $p \equiv 1-t \mod t^2$, then $a^{\frac{p + t - 1}{t^2}}$ is a $t$th root of $a$ whenever $a$ is a perfect $t$th power in $\F_p$. Thus there is a polynomial of degree $\frac{1}{t^2}\cdot p + O(1)$ computing $t$th roots.
This matches the trivial lower bound of $\frac{p-1}{t^2}$ on the degree of polynomials computing $t$th root (proved by counting zeroes of the nonzero polynomial $f(X)^t - X$).

An immediate generalization of Theorem~\ref{thm:1mod4robust} shows that for all other $p$ (namely, $p \not\equiv 1 - t \mod t^2$, but we still preserve the condition that $p \equiv 1 \mod t$),
any polynomial of degree $d$ computing $t$th roots in $\F_p$ with error $e \leq \frac{t-1}{t^2(t+1)} \cdot (p-1) -1$ must satisfy
$$ d \geq \frac{2}{t (t+1) }\cdot (p-1) - e.$$
This is $\frac{2t}{t+1}$ times (which is about double for large $t$) the trivial lower bound, but quite far from the obvious upper bound (from Lagrange interpolation) of $\frac{1}{t}\cdot (p-1)$.

The best upper bound we know for $p$ not of the special form $p \equiv 1 - t \mod t^2$ is for $p$ such that $2p \equiv 2 - t \mod t^2$ (there are infinitely many such $p$ for any odd $t$), and in this case the polynomial $f(X) = X^{\frac{2p+t-2}{t^2} }$ computes $t$th roots. This is of the form
$\frac{2}{t^2} \cdot p + O(1)$, and thus somewhat close to our lower bound for large $t$.

Closing these gaps seems like a very basic and interesting open question.

\section*{Acknowledgements}

Both authors acknowledge support from IAS in 2018--19, where initial
discussions towards this paper took place. We thank N. Carella and Igor Shparlinski for valuable comments and pointers to the literature.

\bibliographystyle{alpha} 
\bibliography{bibfile}

\end{document}